\documentclass{amsart}

\usepackage{amssymb}
\usepackage{amsfonts}
\usepackage{amsmath}

\newtheorem{defn}{Definition}
\newtheorem{prop}{Proposition}
\newtheorem{ex}{Example}

\newcommand{\2}{\int_0^\infty}
\newcommand{\e}{ xe^{-x^2s^2}}
\newcommand{\R}{\mathbb{R}}

\author{Todd Gaugler}
\title{Applications of the $\mathcal{L}_2$-transform to Partial Differential Equations}
\date{\today}

\address{Todd Gaugler 82-15 57th Ave Elmhurst NY 11373} 
\email{gauglertodd@gmail.com}

\begin{document}
\begin{abstract}
This paper aims to demonstrate the applicability of the $\mathcal{L}_2$-integral transform to Partial Differential Equations (PDEs). Of special interest is section (6), which contains an application of the $\mathcal{L}_2$-transform to a PDE of exponential squared order, but not of exponential order. Sections (1) and (2) aim to introduce the history and some elementary properties of the $\mathcal{L}_2$ transform, (3) and (4) include some of the transform's simple applications, and section (5) introduces the $\mathcal{L}_2$ convolution. 
\end{abstract}
\maketitle

\section{Introduction}

Integral transforms, such as the Fourier and Laplace transforms, have numerous applications to differential and partial-differential equations. For example, the Fourier integral transform can be used to solve the popular heat equation, variations of which form the basis for the partial differential equation that leads to the Black-Scholes model for options pricing in Europe. Alternatively, the Laplace transform has applications to physics in the analysis of linear-time invariant systems, like electrical circuits and mechanical systems.  The $\mathcal{L}_2$-transform already has some well understood applications to classical ordinary differential equations such as Bessel's differential equation \cite{YW}(Yurekli, Wilson (2002)) and Hermite's differential equation\cite{YB}(Yurekli, Wilson (2003)). The goal of this article is to illustrate the applicability of the $\mathcal{L}_2$-transform to solving certain partial differential equations-an application of the $\mathcal{L}_2$-transform never before examined in literature. Of particular interest in this paper is an example of a partial differential equation whose solution can be found by using the $\mathcal{L}_2$-transform, but is not of exponential order. This fact is significant in the sense that the nature of the solution assures us that it could not have been achieved through use of the Fourier transform, which can only handle equations of up to exponential order. 

\section{Properties and Definitions Related to the $\mathcal{L}_2$ Transform} 
Recall that a function $f : [0,\infty) \to \R$ is called \emph{exponential squared order} if 
\[\lim\limits_{x \to \infty} f(x) e^{-x^2} = 0 \]

\begin{defn} For any exponential squared order function $f(t)$, the $\mathcal{L}_2$-transform of $f$ is defined as:
\[
\mathcal{L}_2 \{ f(x);s\} = \int_0^\infty xe^{-x^2s^2} f(x) dx 
\]
\end{defn}

\begin{ex} For $n\geq 0$ the following is true:

\begin{equation} \label {x2n}
\mathcal{L}_2 \{ x^{2n};s  \} = \frac {n!}{2s^{2n+2} }  \
\end{equation}
In the case where $n=0$, the definition of the $\mathcal{L}_2$-transform gives us the following:
\[
\mathcal{L}_2\{1;s\}= \2\e  dx = \frac{1}{2s^2}
\]
When $n=1$, we arrive at the following through integration by parts:
\[
\mathcal{L}_2 \{ x^{2n};s  \} = \left( \frac{-x^2}{2s^2}e^{-x^2s^2} \right) \Big |^\infty_0 +\frac{1}{s^2}\2e^{-x^2s^2}xdx
= \frac{1}{s^2}\left( \frac{1}{2s^2}\right)
\]
Property \eqref{x2n} then follows by a simple induction.
\end{ex} 
There exists a differential operator $\delta_x$, defined as follows:
\[ 
\delta_x= \frac{1}{x} \cdot \frac{d}{dx}
\]
We now recall some of the properties of the $\mathcal{L}_2$-transform. 
( \cite{YS}, \cite{YW}).
\begin{prop} Let $f$ be a function of exponential squared order. For all $n\geq 0$, 
\begin{equation} \label{toperator}
\mathcal{L}_2  \{  \delta_{x}f(x);s \} = 2 s^{2} \mathcal{L}_2 \{ f(x) ;s \} -f(0^+)
\end{equation}
where $f(0^+) = \lim_{x^+\to 0}f(x) $, and

\begin{equation}          \label{t2n}
\mathcal{L}_2  \{  x^{2n}f(x);s \} = \frac{(-1)^{n}}{2^n} \delta_s^n \mathcal{L}_2 \{ f(x) ;s \} .
\end{equation}
\end{prop}

\begin{proof}
For the first claim, we calculate
\[
\2\e \delta_x f(x)dx= \2 \e\cdot \frac{1}{x} \frac{d}{dx} f(x)(dx)=\2e^{-x^2s^2}f'(x)dx.
\]
Integrating by parts, we get
\[
f(x)(e^{-x^2s^2} ) \Big |^\infty_0 + \22xs^2 \e f(x)dx.
\]
Evaluating from 0 to $\infty$ and using the fact that $f$ is of exponential squared order, we can write this expression in terms of the $\mathcal{L}_2$-transform:
\[
2s^2\mathcal{L}_2\{f(x);s\} -f(0^+).
\]

For property ($\ref{t2n}$) , taking the case in which n=1, we get:
\[
\delta_s \mathcal{L}_2\{f(x);s \} = \delta_s \2 \e f(x) dx
\]
We can bring the differential operator $\delta_s$ inside this integral, we obtain: 
\[
\2  \left( \frac{1}{s} \right) \frac{d}{ds} \cdot \e f(x) dx = 
 -2 \2  x^2 \cdot \e f(x) dx
\]
We remark that the last term of this equation is equal to:
\[
\mathcal{L}_2 \{-2 x^2 f(x); s\} 
\]
For $n\geq 1$, through induction we arrive at
\[
\delta_s^n \mathcal{L}_2 \{ f(x) ;s \} = -2^n \mathcal{L}_2\{x^{2n} f(x) ; s \} 
\]
Written alternatively,
\[
\mathcal{L}_2  \{  x^{2n}f(x);s \} = \frac{(-1)^{n}}{2^n} \delta_s^n \mathcal{L}_2 \{ f(x) ;s \} .
\]
\end{proof}

\section{A First Application}

Consider the following partial differential equation :

\begin{equation} \label{thePDE}
t^3 u_{tx} + 2xu = 0 \quad \text{and} \quad u(0^+,t) = 0
\end{equation}
where $u = u(x,t)$ for $x,t >0$ and $u(0^+,t) = \lim\limits_{x \to 0^+} u(x,t)$.  

Writing the PDE in Equation (\ref{thePDE}) in terms of the differential operator, we have: 
\[
t^3\frac{1}{x}\frac{d}{dx}u_t+2u =0 
\]
or equivalently,
\[
 t^3\delta_x u_t = -2u.
\]
Taking the $\mathcal{L}_2$-transform of both sides and using property \eqref{t2n} we obtain
\[
2s^2t^3 \hat{u}_t-u(0^+,t)=-2\hat{u} 
\]
where $\hat u = \hat u(s,t)  = \mathcal{L}_2 \{ u(x,t) ;s \}$. Equivalently,
\[
 \hat{u}_t = \frac{-1}{s^2t^3}\hat{u}+\frac{u(0^+,t)}{2s^2t^3}.
\]
By the initial condition in (\ref{thePDE}), this last term is zero.  
One solution to this differential equation is 
\[
\hat u(s,t) = \frac{1}{2s^2}e^{{\frac{t^{-2}s^{-2}}{2}}} 
\]

We now write $\hat u(s,t)$ as a series to obtain:
\[
\hat u(s,t)= \frac{1}{2s^2} \displaystyle\sum\limits_{n=0}^\infty \frac{(\frac{t^{-2}s^{-2}}{2})^{n}}{n!} = \displaystyle\sum\limits_{n=0}^\infty \frac{\frac{1}{2}^{n}}{2s^{2n+2}\cdot t^{2n}\cdot n!}.
\]
Using property \eqref{x2n}, we can calculate that $u(x,t) = \mathcal{L}_2^{-1}\{(\hat u(s,t) ; x\}$ is
\[
u(x,t) = \displaystyle\sum\limits_{n=0}^\infty \frac{\frac{1}{2}^{n} x^{2n}}{t^{2n}( n!)^{2}}.
\]

\section{Further Generalizations}
Similarly, we can solve all partial differential equations of the form:
\begin{equation}
0=f(t)u + f(t)\frac{1}{x}u_x+g(t)\frac{1}{x}u_{xt} \label{PDE2}
\end{equation}
Taking the $\mathcal{L}_2$-transform of the PDE in (\ref{PDE2}), we compute:
\[
0=f(t)\hat{u} + f(t)2s^2\hat{u}+g(t)2s^2\hat{u}_t
\]
Rearranging:
\[
\frac{\hat{u}_t}{\hat{u}} = -\frac{(1+2s^2)}{2s^2}\cdot M(t)
\]
Where $M(t) = \frac{f(t)}{g(t)}$. Claiming that $L(s)$ is any function of $s$, a solution for $\hat{u}$ is as follows:
\[
\hat{u}(s,t)=e^{\frac{-(1+2s^2)}{2s^2}\cdot \int_0^t{M(w)dw}}\cdot L(s)=e^{ \int_0^t{M(w)dw}} \cdot e^{\frac{-\int_0^t{M(w)dw}}{2s^2}}\cdot L(s)
\]
Writing this as a series--denoting $ \int_0^t{M(w)dw}$ as $\mathcal{M}(t)$,and letting $L(s)=\frac{1}{-s^2}$--we get:
\[
\hat{u} = \frac{1}{-s^2} \displaystyle\sum\limits_{n=0}^\infty \frac{-\mathcal{M}(t)^n}{n! \cdot (2s^2)^n}\cdot \displaystyle\sum\limits_{n=0}^\infty \frac{\mathcal{M}(t)^n}{n!}= \displaystyle\sum\limits_{n=0}^\infty \frac{\mathcal{M}(t)^n}{n! \cdot (2^n s^{n+2})}\cdot \sum\limits_{n=0}^\infty \frac{\mathcal{M}(t)^n}{n!}
\]

\[
= \sum \limits_{n=0}^\infty \frac{\mathcal{M}(t)^n}{n! \cdot 2^{(n-1)} \cdot (2s^{n+2}) } \cdot \sum \limits_{n=0}^\infty \frac{\mathcal{M}(t)^n}{ n!}
\]
using identity ($\ref{toperator}$), and taking the inverse $\mathcal{L}_2$-transform, we claim that:
\[
u(x,t)=\sum \limits_{n=0}^\infty  \frac{\mathcal{M}(t)^n \cdot x^{2n} }{(n!)^2 \cdot 2^{n-1}} \cdot \sum \limits_{n=0}^\infty \frac{\mathcal{M}(t)^n}{2 n!} 
\]
Lastly, we can always solve the following types of partial differential equations:
\begin{equation} \label{pde3}
0=f(t)u+f(t)\frac{1}{x}u_x+ g(t)u_t+ g(t)u+{xt}
\end{equation}
Taking the $\mathcal{L}_2$-transform of PDEs of the firm shown in (\ref{pde3}), we get:
\[
0=\hat{u}(1+2s^s)f(t) + \hat{u}_t(1+2s^2)g(t) \Rightarrow \frac{\hat{u}_t}{\hat{u}} = J(t)
\]
Where $J(t)= -\frac{g(t)}{f(t)}$. This implies that  a solution for $\hat{u}$ is as follows:
\[
\hat{u}(s,t) =e^{\int_0^{t} J(w)dw}\cdot L(s)
\]
Where L(s) is an arbitrary function of $s$. Representing this expression as a series, and letting $L(s)= \frac{1}{2s^2}$, we get:
\[
\hat{u}=\sum_{n=0}^{\infty} \frac{(\int_0^{t} J(w)dw)^n}{n! \cdot 2y^2}
\]
And after applying the $\mathcal{L}_2$ inverse to our sum, we get the following solution for $u(x,t)$:
\[
u(x,t)=x^0 \cdot e^{\int_0^{t} J(w)dw)} 
\]

\section{The $\mathcal{L}_2$ Convolution}
\begin{defn}
A binary operation ($\star$) called the {\bfseries convolution} of two functions $f$ and  $g$ is defined as follows:
\[
(f\star g)(t) = \int_0^{t}xf(\sqrt{t^2-x^2})g(x)dx
\]
\end{defn}
It can be shown that this operation is associative and commutative. Most importantly for our purposes, the following are true:
\[
\mathcal{L}_2 \{f\star g;s \}= \mathcal{L}_2(f) \cdot \mathcal{L}_2(g)
\]
and:
\begin{equation}\label{l2property1}
\mathcal{L}_2\{f_1\star f_2\star ... \star f_n; s\} = \mathcal{L}_2(f_1)\cdot  \mathcal{L}_2(f_2)\cdot...\cdot\mathcal{L}_2(f_n)  
\end{equation}
From which it follows that:
\begin{equation}\label{l2property2}
\mathcal{L}^{-1}\{ \hat{f}_1\cdot \hat{f_2}\cdot...\cdot \hat{f}_n\}= f_1\star f_2 \star ...\star f_n  
\end{equation}

\section{An Application to a Partial Differential Equation that is of Exponential Squared Order, and not of Exponential Order}
Consider the following partial differential equation, with the following condition:
\begin{equation}\label{lastex}
0=g(t)u- f(t)u_{t}+ \frac{1}{x}f(x)u_{xt} \quad \quad u(0^+,t)=0
\end{equation}
Writing this equation in terms of the differential operator,
\[
0=g(t)u-f(t)u_t+\delta_x f(t)u_t
\]
and applying the $\mathcal{L}_2$-transform, we arrive at the following expression:
\[
0=g(t)\hat{u}_t-f(t)\hat{u}_t+2s^2f(t)\hat{u}_t
\]
Rearranging,
\[
0=g(t)u+f(t)\hat{f}_t(-1+2s^2)
\]
Which leads to:
\[
\frac{\hat{u}_t}{\hat{u}} = -\frac{1}{-1+2s^2}\cdot H(t)
\]
Where $H(t)=\frac{g(t)}{h(t)}$. This implies that a solution for $\hat{u}$ is as follows:
\[
\hat{u}(s,t)= e^{\frac{-1}{-1+2s^2} \cdot \int_0^tH(w)dw }
\]
Writing this solution as a series, we get:
\[
\hat{u}(s,t)= \sum_{n=0}^{\infty} \frac{ (\int_0^t H(w)dw)^n }{n!} \cdot \left(\frac{1}{-1+2s^2}\right)^n
\]
from $\eqref{l2property1}$, $\eqref{l2property2}$ and the following identity:
\[
\mathcal{L}_2\{e^{ax^2} ; s\} = \frac{1}{-2a+2s^2}
\]
We can arrive at the following solution:
\[
\mathcal{L}_2^{-1} \{\hat{u}(s,t) \} =1+ \sum_{n=1}^\infty \left( \frac{(\int_0^t H(w)dw)^n}{n!} \cdot (e^{(1/2)x^2})^{\star_n} \right)
\]
Where $(e^{(1/2)x^2})^{\star_n}$ represents the following:
\[
\underbrace{ e^{(1/2)e^2} \star e^{(1/2)e^2} ... \star e^{(1/2)e^2} }_{n}
\]
Which leads to the following proposition, which can be shown through induction: 
\[
(e^{(1/2)e^2})^{\star_n}= \frac{1}{2^{n-1}}\cdot \frac{ x^{2(n-1)} \cdot e^{(1/2)x^2}}{(n-1)!}
\]
For $n\geq1$. Applying this proposition to the cases in which $n\geq 1$, we have the following solution for $u(x,t)$:
\[
u(x,t)= 1 +\sum_{n=1}^\infty \left( \frac{(\int_0^t H(w)dw)^n}{n!} \cdot  \frac{1}{2^{n-1}}\cdot \frac{ x^{2(n-1)} \cdot e^{(1/2)x^2}}{(n-1)!} \right)
\]
We now make some remarks regarding the order of our solution. For the sake of simplicity, we assume that our function $\int_0^t H(w) dw= 1$, giving us the following:
\[
u(x,t) =1+\sum_{n=1}^\infty \left( \frac{ x^{(2n -2)}\cdot e^{ \frac{x^2}{2} } }{n!\cdot 2^{n-1} \cdot (n-1)! }\right)
\]
Notice that:
\[
u(x,t) =1+ \sum_{n=1}^\infty \left( \frac{ x^{(2n -2)}\cdot e^{ \frac{x^2}{2} } }{n!\cdot 2^{n-1} \cdot (n-1)! }\right)
=
1+e^{ \frac{x^2}{2} }\cdot  \sum_{n=0}^{\infty}\frac{x^{2n}\cdot}{(n+1)\cdot(n!)^2\cdot 2^{n} },
\]
and that 
\[
\sum_{n=0}^{\infty}\frac{x^{2n}\cdot}{(n+1)\cdot(n!)^2\cdot 2^{n} }
\]
is a sum $\mathcal{S}$ of elements in $\mathbb{R}^+$, where one can take the sum of first two terms where $n=0, 1$ as evidence that $\mathcal{S}\geq 1$. 
From this it follows clearly that 
\[
e^{\frac{x^2}{2}} \cdot \mathcal{S} \geq  e^{\frac{x^2}{2} } ,
\]
and since
\[
\lim_{x \to \infty} e^{-x} (1+e^{\frac{x^2}{2} }) =\infty,
\]
our solution is not of exponential order. Alternatively, notice that 
\[
(n+1)\cdot(n!)^2 \cdot 2^n \geq (n!) \cdot 2^n \cdot 2^n ,
\]
because ultimately, 
\[
n! \geq a^n
\]
for some fixed $a\in \mathbb{R}$. From this inequality we conclude that:
\[
e^{ \frac{x^2}{2} }\cdot  \sum_{n=0}^{\infty}\frac{x^{2n}\cdot}{(n+1)\cdot(n!)^2\cdot 2^{n} } \leq e^{ \frac{x^2}{2}} \cdot \sum_{n=0}^\infty \frac{ x^{2n} }{ n! \cdot 4^n } = e^{ \frac{x^2}{2} } \cdot \sum_{n=0}^\infty \frac{ \left( \frac{x^2}{4} \right)^n }{n!} = e^{ \frac{x^2}{2} } \cdot e^{ \frac{ x^2}{4} }  = e^{ \frac{3x^2}{4} }
\]
Noticing that 
\[
\lim_{x\to \infty} e^{-x^2}(1+e^{\frac{3x^2}{4 }}) = \lim_{x\to \infty}\left[e^{-x^2}+e^{ \frac{ -x^2}{4} }\right] = \lim_{x\to \infty} e^{-x^2} + \lim_{x\to \infty} e^{\frac{-x^2}{4} } =0,
\]
We see that our solution is in fact, of exponential squared order. Since our solution was not of exponential order, one would not have been able to calculate the solution to our PDE in equation (\ref{lastex}) with a Fourier integral transform, which can only handle equations up to exponential order.


\begin{thebibliography}{99}
\bibitem{YS} O. Yurekli, I. Sadek, A Parseval-Goldstein type theorem on the Widder potential transform 
and its applications, Int. J. Math. Math. Sci. 14 (1991) 517-524.
\bibitem{YW} O. Yurekli, S. Wilson, A new method of solving Bessel's differential equation using the $\mathcal{L}_2$-transform,
Appl. Math. and Comp., 130 (2002) 587-591.
\bibitem{YB} O. Yurekli, S. Wilson, A new method of solving Hermite's differential equation using the $\mathcal{L}_2$-transform,
Appl. Math. and Comp., 145 (2003) 495-500 
\end{thebibliography}
\end{document}